\documentclass[secnum,leqno]{amsart}
\usepackage{amssymb, amsmath}
\newtheorem{thm}{Theorem}[section]

\newtheorem{lmm}[thm]{Lemma}

\newtheorem{conjecture}[thm]{Conjecture}

\theoremstyle{definition}

\theoremstyle{remark}
\newcommand{\C}{{\mathbb{C}}}

\newcommand{\R}{{\mathbb{R}}}
\newcommand{\HH}{{\mathbb{H}}}
\newcommand{\Z}{{\mathbb{Z}}}

\newcommand{\qtq}[1]{\quad\text{#1}\quad}
\newcommand{\eps}{\varepsilon}

\DeclareMathOperator{\tr}{tr}
\DeclareMathOperator{\dist}{dist}

\DeclareMathOperator{\He}{He}

\let\Re=\undefined\DeclareMathOperator{\Re}{Re}
\let\Im=\undefined\DeclareMathOperator{\Im}{Im}
\let\det=\undefined\DeclareMathOperator{\det}{det}

\DeclareMathOperator{\E}{\mathbb{E}}
\DeclareMathOperator{\PP}{\mathbb{P}}

\newcommand{\Sing}{\mathcal{S}}

\begin{document}

\title[Sonin's argument, the shape of solitons, and the most stably singular matrix]{Sonin's argument, the shape of solitons,\\and the most stably singular matrix}

\author{Rowan Killip and Monica Vi\c{s}an}

\address
{Rowan Killip\\
Department of Mathematics\\
University of California, Los Angeles, CA 90095, USA}
\email{killip@math.ucla.edu}

\address
{Monica Vi\c{s}an\\
Department of Mathematics\\
University of California, Los Angeles, CA 90095, USA}
\email{visan@math.ucla.edu}

\begin{abstract}
We present two adaptations of an argument of Sonin, which is known to be a powerful tool for obtaining both qualitative and quantitative information about special functions; see \cite{Szego}.  Our particular applications are as follows:

(i) We give a rigorous formulation and proof of the following assertion about focusing NLS in any dimension:  The spatial envelope of a spherically symmetric soliton in a repulsive potential is a non-increasing function of the radius.

(ii) Driven by the question of determining the most stably singular matrix, we determine the location of the maximal eigenvalue density of an $n\times n$ GUE matrix.  Strikingly, in even dimensions, this maximum is \emph{not} at zero.
\end{abstract}

\maketitle

\section{Introduction}

We consider two questions in this paper: one taken from the study of dispersive PDE, the other from random matrix theory.  The common feature is that both will be tackled by adapting an elegant argument of Sonin \cite{Sonin}.

We learned of Sonin's argument from Szeg\H{o}'s book \cite{Szego}, where it forms the subject of Section~7.31 (see also the prefatory remarks to Chapter~VII).  We quote here verbatim the statement of Theorem~7.31.1 from \cite{Szego}:

\begin{thm}\label{T:S}
Let $y=y(x)$ satisfy the differential equation
$$
y'' + \phi(x) y =0,
$$
where $\phi(x)$ is a positive function having a continuous derivative of a constant sign in $x_0<x<X_0$.  Then  the successive relative maxima of $|y|$, as $x$ increases from $x_0$ to $X_0$, form an increasing or decreasing sequence according as $\phi(x)$ decreases or increases.
\end{thm}

The proof is ingeniously simple: The function
$$
f(x) = y(x)^2 + \tfrac{1}{\phi(x)} y'(x)^2
    \qtq{satisfies} f'(x) = - \tfrac{\phi'(x)}{\phi(x)^2}\, y'(x)^2
$$
and so has opposite monotonicity to that of $\phi$.  On the other hand, $f(x)=y(x)^2$ at each local extremum of $y(x)$.  Thus, the theorem is proved.

Theorem~\ref{T:S} provides a powerful tool for understanding the overall shape of the classical special functions as can been seen already from the applications detailed in \cite{Szego}.  Indeed, it is worthy of note that this is the argument of choice, despite the fact that the classical special functions admit a wealth of series and integral representations.

In Section~\ref{S:2} we apply a similar argument to describe (at least qualitatively) the spatial envelope of solitons for NLS. In Section~\ref{S:3}, we determine the point of highest density in the eigenvalues of a GUE matrix.  In particular, we will see that this point is non-zero in even dimensions, which shows that zero is not the most stably singular matrix.  These two sections may be read independently of one another.  With this in mind, we leave more precise formulations, together with the necessary preliminaries, to the individual sections.  In closing, however, we would like present an application of Theorem~\ref{T:S} that requires no prerequisites, namely, to rigorously demonstrate the well-known decaying envelope overlaying the oscillatory behaviour of Bessel functions.

The Bessel function $J_0(x)$ is defined as the unique solution to
\begin{equation}\label{E:J0}
x^2 y''(x) + x y'(x) + x^2 y(x)=0 \qtq{with} y(0)=1.
\end{equation}
This ODE has a regular singular point at $x=0$; any solution linearly independent of $J_0(x)$ is unbounded near $x=0$.

It is elementary to verify from \eqref{E:J0} that
\begin{equation}\label{E:J0'}
y(x) = J_0(e^x) \qtq{solves} y''(x) + e^{2x} y(x) =0
\end{equation}
and that
\begin{equation}\label{E:J0''}
y(x) = \sqrt{x} J_0(x) \qtq{solves} y''(x) + \bigl( 1 + \tfrac{1}{4x^2}\bigr) y(x) =0 \quad \text{for $x>0$}.
\end{equation}

Applying the Sturm comparison theorem (cf. \cite[Ch. 8]{CodLev}) to either \eqref{E:J0'} or \eqref{E:J0''} shows that the Bessel function changes sign infinitely many times on the positive axis.  On the other hand, applying Theorem~\ref{T:S} to \eqref{E:J0'}, we see that these oscillations are decaying in magnitude, as measured, for example, by the size of extrema between successive zeros.  As a counter point, however, we see by applying Theorem~\ref{T:S} to \eqref{E:J0''}, that $J_0(x)$ is not $o(x^{-1/2})$ as $x\to\infty$.

\section{The shape of solitons}\label{S:2}

Solutions of the form $\psi(t,x)=u(x)e^{i\omega t}$ to
\begin{align}\label{NLS psi}
i\partial_t \psi = -\Delta \psi + V \psi - |\psi|^p\psi  
\end{align}
are commonly termed solitons, at least when $u(x)$ decays appropriately at infinity.  In this paper, we follow the widely adopted practice of requiring that $u\in H^1(\R^d)$.  For ease of exposition, we shall assume throughout that $V\in C^\infty(\R^d)$.  In particular, it follows that $\psi\in C^\infty$.

The goal of this section is to show that if $u(x)$ (and so also $V(x)$) is spherically symmetric and $V(x)$ is repulsive, in the sense that $x\cdot \nabla V(x)\leq 0$, then the envelope of $u(x)$ is a decreasing function of radius.  

Due to the assumption of spherical symmetry, the claim just made reduces to a result about ordinary differential equations.  Concretely, $\psi(t,x) =e^{i\omega t} y(|x|)$ is a solution to \eqref{NLS psi} if and only if $y$ is a smooth solution to
\begin{align}\label{NLSODE}
-y''(r) - \tfrac{d-1}{r} y'(r) + V(r) y(r) - |y(r)|^p y(r) = - \omega y(r) \qtq{with} y'(0)=0,
\end{align}
where we agree to write $V(x)=V(|x|)$, in line with the fact that $V$ is spherically symmetric.  The fact that $V$ is repulsive may now be written as $V'(r)\leq 0$ for $r\geq 0$.

While $u$ and so $y$ may be complex-valued, in principle, let us now observe that we may always reduce matters to the real-valued case.  Multiplying \eqref{NLSODE} by $r^{d-1}\bar y(r)$ and then taking imaginary parts yields
$$
\tfrac{d\ }{dr} \Im \{ r^{d-1} \bar y(r) y'(r) \} = 0.
$$ 
On the other hand, as $u\in H^1(\R^d)$, we know that $\Im \{ r^{d-1} \bar y(r) y'(r)\} \to 0$ as $r\to\infty$ (at least along some subsequence).  Thus $\Im\{ \bar y(r) y'(r)\}\equiv0$.  This then implies that $y$ can be written as a real-valued function multiplied by a uni-modular complex number.  (We use here that $y$ and $y'$ cannot vanish simultaneously without forcing $y\equiv 0$.)  This complex number can then be factored out of \eqref{NLSODE} leaving us to consider only real-valued solutions to \eqref{NLSODE}.

The main result of this section is the following (we also discuss two further applications at the end of the section):

\begin{thm}\label{T:ode}  Let $y:[0,\infty)\to\R$ be a solution to \eqref{NLSODE} with $V$ repulsive.  Then the successive local maxima of $|y(r)|$ form a non-increasing sequence as $r$ increases over the interval $[0,\infty)$. 
\end{thm}

\begin{proof}
Let $r_k$ denote the locations of the successive local maxima of $|y(r)|$, which we enumerate consecutively outward from the origin.  As extrema, we have
$$
y'(r_k) =0 \qtq{and} y(r_k) y''(r_k) \leq 0.
$$
Plugging this information into \eqref{NLSODE} we deduce that
\begin{align}\label{2:33}
|y(r_k)|^{p+2} \geq [ V(r_k) +  \omega ]  y(r_k)^2.
\end{align}

Consider now
\begin{align*}
f(r):= \tfrac12 \bigl[y'(r)\bigr]^2 - \tfrac12 V(r) y(r)^2 + \tfrac1{p+2}\bigl|y(r)\bigr|^{p+2} - \tfrac\omega2 y(r)^2 ,
\end{align*}
which obeys
\begin{align*}
f'(r) = - \tfrac{d-1}{r} \bigl[y'(r)\bigr]^2 - \tfrac12 V'(r) y(r)^2 .
\end{align*}

Assuming, toward a contradiction, that $y(r_{k+1})^2\geq y(r_k)^2$, we obtain
\begin{align*}
f(r_{k+1}) - f(r_k) \leq \int_{r_k}^{r_{k+1}} - \tfrac12  V'(r) y(r)^2 \,dr \leq \tfrac12  \bigl[ V(r_k) - V(r_{k+1})\bigr] y(r_{k+1})^2,
\end{align*}
which may then be rearranged to reveal
\begin{align}
g_k\bigl(y(r_k)\bigr) \geq g_k\bigl(y(r_{k+1})\bigr)  \qtq{where} g_k(y) := \tfrac1{p+2}\bigl|y\bigr|^{p+2} - \tfrac12\bigl[V(r_k)  + \omega \bigr] y^2 .
\end{align}

In order to reach a contradiction and so complete the proof, it remains to show that $g_k(y)$ is an increasing function over the interval $[|y(r_k)|,\infty)$, which contains $|y(r_{k+1})|$ due to our contradiction hypothesis.
This is easily achieved using \eqref{2:33}:
\begin{align*}
y\partial_y g_k(y) & = \bigl|y\bigr|^{p+2} - \bigl[V(r_k)  + \omega \bigr] y^2 \\
&\geq \frac{y^2}{y(r_k)^2} \Bigl[ |y(r_k)|^{p+2} - [ V(r_k) +  \omega ]  y(r_k)^2 \Bigr] \geq0,
\end{align*}
whenever $|y| \geq |y(r_k)|$.
\end{proof}

Thus far, our discussion has focused on spherically symmetric solutions.  As we will now explain, Theorem~\ref{T:ode} can also be applied to a wider class of solutions in dimensions one and two (where there are non-trivial spherical harmonics of constant modulus).  We retain the requirements that $V$ be symmetric and repulsive.

Let us first consider the case $d=1$ and let $y:\R\to\R$ be an \emph{odd} solution to \eqref{NLSODE}.  Evidently, $r=0$ is not a local maximum of $y$.  Let $r_0>0$ denote the first such local maximum.  Then we may apply Theorem~\ref{T:ode} to the function $\tilde y(r) = y(r+r_0)$ with attendant potential $V(r+r_0)$ and so discover that the successive local maxima of $|y(r)|$ remain a non-increasing function of radius (over the whole range $r\geq 0$) in the case of odd solutions on the line.

Turning now to $d=2$, we adopt polar co-ordinates: $x=(r\cos(\theta),r\sin(\theta))$ and consider solutions to \eqref{NLS psi} of the form $\psi(t,x)=y(r)e^{i \ell\theta + i\omega t}$ where $\ell\in\Z$.  Simple computations show that for such a function to be a solution to \eqref{NLS psi}, one must have
\begin{align*}
-y''(r) - \tfrac{d-1}{r} y'(r) + \bigl[V(r)+\tfrac{\ell^2}{r^2} + \omega\bigr]  y(r) - |y(r)|^p y(r) = 0.
\end{align*}
Moreover, unless $\ell=0$ (which was treated already), we must have $y(r)\to 0$ as $r\downarrow 0$.  In this way, we find ourselves in the setting of the previous paragraph: we choose $r_0>0$ as the radius of the first local maximum of $|y(r)|$ and apply Theorem~\ref{T:ode} to $\tilde y(r) = y(r+r_0)$.  Once again, we find that the envelope of such solutions is a decreasing function of radius, for all $r>0$.

\section{The most stably singular matrix}\label{S:3}

There can be little argument that the zero matrix is the `most singular' of all symmetric matrices.  In this section, we discuss how this can fail under the addition of Gaussian noise.  

The most natural notion of Gaussian noise is that adopted in random matrix theory, which we will now describe. See \cite{Mehta} for an alternate introduction and further discussion.

The sets of $n\times n$ real-symmetric matrices, complex hermitian matrices, and quaternion self-dual matrices, are all vector spaces over $\R$ and all admit a natural inner product:
$$
\langle A, B\rangle = \Re \tr ( A^\dagger B ),
$$
where $\dagger$ denote the transpose, the hermitian conjugate, or the quaternionic dual, as appropriate.   Having fixed a base field and a size $n$, let $\{E_j\}$ denote an orthonormal basis (over $\R$) for the associated space and then define
$$
X = \sum_j Z_j E_j
$$
where $Z_j\sim N(0,1)$ denote independent standard Gaussian random variables.  The law of this random matrix $X$ is said to be GOE when the base is $\R$. It is called GUE and GSE when the base field is $\C$ and $\HH$, respectively.

We may now be more precise about the question we wish to tackle in this section:  For what deterministic matrix $A$ is
\begin{equation}\label{E:M}
M = A + X
\end{equation}
most likely to be singular?  It is natural to place a constant in front of $X$ to represent the size of the noise; however, this can always be scaled away.

Of course, $A + X$ has zero probability of being singular because the set of singular matrices
$$
\Sing=\{B : \det B =0\}
$$
is a variety of co-dimension one.  (For the theory of determinants over $\HH$, see \cite{DysonQdet}.) Correspondingly, we interpret the likelihood of being singular through the Radon--Nikodym derivative; see \eqref{E:L:RN'} below.

As in \cite{Mehta}, the \emph{one-point function}, or density of eigenvalues, of the random matrix $M$ will be denoted $R^{(1)}(x)$.  It is uniquely determined by the relation
$$
\int_\R   p(x)  R^{(1)}(x)\,dx = \E \tr\{ p(M) \}, \quad\text{for all polynomials $p$}.  
$$
Note that $\int R^{(1)}(x)\,dx = n$, the number of eigenvalues.

\begin{lmm}\label{L:RN}
\begin{equation}\label{E:L:RN}
    \PP\bigl\{ \dist(M,\Sing) \leq \eps\bigr\} =  \PP\bigl\{ \|M^{-1}\|_{\text{\upshape op}} \geq \eps^{-1} \bigr\} = \int_{-\eps}^\eps R^{(1)}(x)\,dx  + o(\eps)
\end{equation}
and correspondingly,
\begin{equation}\label{E:L:RN'}
    \lim_{\eps\to 0} \tfrac{1}{2\eps} \PP\bigl\{ \dist(M,\Sing) \leq \eps\bigr\} =  R^{(1)}(0).
\end{equation}
\end{lmm}

\begin{proof}
The first equality in \eqref{E:L:RN} follows from the elementary fact that 
\begin{equation}
    \dist(B,\Sing) = \inf \{ |\lambda| : \lambda \text{ is an eigenvalue of $B$}\}.
\end{equation}
(See also the Hoffman--Wielandt inequality \cite{MR0052379}.)

Regarding the second equality in \eqref{E:L:RN}, we first note that
\begin{equation}\label{11:33}
\PP\bigl\{ \dist(M,\Sing) \leq \eps\bigr\} \leq  \int_{-\eps}^\eps R^{(1)}(x)\,dx,
\end{equation}
since RHS\eqref{11:33} represents the average number of eigenvalues in the interval $[-\eps,\eps]$.   Moreover, we see that it suffices to control the probability that such an interval contains two or more eigenvalues in order to obtain an inequality in the opposite direction.  The stated bound follows from the simple observation that the space of matrices with two vanishing eigenvalues is a variety (homogeneous with respect to scaling) of higher codimension.  The exact codimension depends on the ambient field, so we settle for the crude bound $o(\eps)$.  
\end{proof}

The principal result of this section is the determination of the location of the maximum of the one-point function $R^{(1)}(x)$ associated to the model \eqref{E:M} when $A=0$ and one is working over the complex field.  Equivalently, we determine the choice of $A$, from among numerical multiples of the identity, that maximizes the probability of being singular.  As we will see, the optimum is \emph{not} the zero matrix when $n$ is even.  Indeed, $x=0$ is a local \emph{minimum} of $R^{(1)}(x)$.  The intuitive explanation lies in the strength of eigenvalue repulsion: the middle eigenvalues each push each other away from the origin.  One expects this effect to be even stronger in the quaternionic case. The effect is weaker in the real-symmetric setting and messy computations (not reproduced here) show $x=0$ is a degenerate local maximum of $R^{(1)}(x)$ in that case.

While we are not aware of any prior works attacking precisely the question posed in this section, the relations \eqref{E:L:RN} and \eqref{11:33} connect this question to matters of on-going interest; see, for example,
\cite{MR3631932,MR3645115,MR3460183}.  Relative to these works, what we are seeking to achieve (in the setting of Gaussian noise) is not only the optimal power dependence on $n$, but even the optimal constant.

The explanation for working over the complex field is the fact that there is an elegant explicit formula for the one-point function; see \cite[\S5.2]{Mehta} and Lemma~\ref{1pt} below.  An explicit formula is also known for the one-point function in the case $M\sim A+GUE$, see \cite{MR2641363}; however, we are not able to handle the case of general $A$ at this time.

\begin{lmm}\label{1pt}
Suppose $M\sim{}$GUE. Then the one-point function is given by
\begin{align*}
R^{(1)}(x) &= \tfrac1{\sqrt{2\pi}} \sum_{k=0}^{n-1} p_k(x)^2 e^{-\frac12 x^2} = \sqrt{\tfrac{n}{2\pi}} \bigl[p'_{n}(x) p_{n-1}(x) - p_{n}(x) p'_{n-1}(x)\bigr] e^{-\frac12 x^2},
\end{align*}
where $p_k$ are rescaled Hermite polynomials,
\begin{equation}\label{HermiteDefn}
p_k(x) := \tfrac{(-1)^k}{\sqrt{k!}} \, e^{\frac12 x^2} \frac{d^k\ }{dx^k} e^{-\frac12 x^2}
    = \tfrac{x^k}{\sqrt{k!}} + \text{lower order}.
\end{equation}
These polynomials are orthonormal with respect to the measure $(2\pi)^{-1/2}e^{-x^2/2}\,dx$.
\end{lmm}

We will need the following basic facts about the Hermite polynomials:
\begin{align}\label{Hode}
- p_k''(x) + x p_k'(x) = k p_k(x) \qtq{and} p_{k}'(x) = \sqrt{k} p_{k-1}(x).
\end{align}
These relations can be found in \cite{MR0167642}, or derived directly from \eqref{HermiteDefn}.  In the notation of \cite{MR0167642}, we have $\He_k(x) = \sqrt{k!}\, p_k(x)$.

To locate the global maximum of $R^{(1)}(x)$, we employ the following Sonin-style lemma:

\begin{lmm}\label{L:Sonin}
Suppose $y''(x) + \phi(x) y(x)=0$ on an interval $[a,b]$ with $\phi\in C^1$ and strictly decreasing.  Then
the values of $|y'(x)|$ at the successive zeros of $y(x)$ form a strictly decreasing sequence.
\end{lmm}

\begin{proof}
If $f(x) = [y'(x)]^2 + \phi(x) [y(x)]^2$, then $f'(x)=\phi'(x) [y(x)]^2 < 0$.
\end{proof}

We are now prepared to prove the main result of this section:

\begin{thm}
For $M\sim{}$GUE, we have
$$
\sup_x R^{(1)}(x) =  R^{(1)}(0) = \tfrac{1}{\sqrt{2\pi}} \frac{(2k+1)!}{2^{2k}[k!]^2}
    = \tfrac1{\pi} \sqrt{n} \, \exp\bigl\{\tfrac1{4n} + O(n^{-3})\bigr\}
$$
 when $n=2k+1$ is odd.  On the other hand, when $n$ is even,
$$
\sup_x R^{(1)}(x) = R^{(1)}(\pm x_n) = \tfrac{n}{\sqrt{2\pi}} p_{n-1}(x_n)^2 e^{-\frac12 x_n^2},
$$
where $x_n$ is the smallest positive zero of $p_{n}$. Moreover, $R^{(1)}(x) < R^{(1)}(\pm x_n)$ for $x\neq \pm x_n$.
\end{thm}

\begin{proof}
From Lemma~\ref{1pt} and \eqref{Hode} we find
$$
\partial_x R^{(1)}(x) = - \sqrt{\tfrac{n}{2\pi}} \, p_{n}(x) p_{n-1}(x) e^{-\frac12 x^2}.
$$

As the zeros of $p_{n}$ and $p_{n-1}$ interlace (cf. \cite[Theorem~3.3.2]{Szego}) and the last critical point of $R^{(1)}(x)$ must be a local maximum (since the one-point function decays at infinity), we see that the local maxima of $R^{(1)}(x)$ occur at the zeros of $p_{n}$, while those of $p_{n-1}$ correspond to local minima. Moreover, Lemma~\ref{1pt} and \eqref{Hode} reveal that 
$$
R^{(1)}(x) =  \tfrac{1}{\sqrt{2\pi}} \bigl| p_{n}'(x) e^{-x^2/4} \bigr|^2 \qtq{whenever} p_n(x) =0.
$$

On the other hand, by \eqref{Hode}, $y_n(x):=p_n(x)e^{-x^2/4}$ satisfies
$$
y_n''(x) + (n+\tfrac12-\tfrac14 x^2) y_n(x)=0.
$$
In this way, the theorem follows by applying Lemma~\ref{L:Sonin} on the interval $[0,\infty)$ and the trivial observation that $p_k(-x)=(-1)^kp_k(x)$. 
\end{proof}

One way to generate a matrix of the form \eqref{E:M} is to begin with the matrix $A$ and allow each of the matrix entries to perform Brownian motion.  A famous calculation of Dyson \cite{DysonBrownian} determines the stochastic process followed by the eigenvalues of $M$ in this setting.  Concretely, the eigenvalues perform the diffusion
\begin{equation}\label{E:DBM}
d\lambda_j = dB_j(t) + \tfrac{\beta}{2} \sum_{k\neq j} \frac{dt}{\lambda_j-\lambda_k}
\end{equation}
where $B_1,\ldots,B_n$ denote independent Brownian motions.  Here, $\beta=1$ for real symmetric matrices, $\beta=2$ for hermitian matrices, and $\beta=4$ in the quaternion
case.  Nevertheless, the diffusion makes perfect sense for any $\beta\geq0$.  We define the one-point function $R^{(1)}_\beta(x)$ associated to such a diffusion at time $t=1$ in the natural way:
\begin{equation}
R^{(1)}_\beta(x) = \E \sum_j \delta(x-\lambda_j(t=1)).  
\end{equation}
Evidently, the function $R^{(1)}_\beta$ depends also on the initial data for the diffusion.  In the setting of \eqref{E:M}, this is given by the eigenvalues of the matrix $A$.  As in the study of \eqref{E:M}, there is no loss of generality in restricting to time $t=1$, since other values can then be recovered by scaling.

The results presented here, together with some further fragmentary evidence, leads us to the following conjecture:
\begin{conjecture}\label{CC}
Among all initial conditions for Dyson Brownian motion \eqref{E:DBM}, those maximizing $R^{(1)}(0)$ are precisely the following:
\begin{align*}
\lambda_1=\cdots=\lambda_n = \begin{cases} 0 & \text{if $\beta\leq 1$ or $n$ is odd,}\\ \pm x_{n,\beta}\neq 0 & \text{if $\beta > 1$ and $n$ is even.} \end{cases}
\end{align*}
\end{conjecture}

We do not have a conjecture about the precise value of $x_{n,\beta}$.  

\section*{Acknowledgements}

The research reported herein was completed during the ``Harmonic Analysis and Partial Differential Equations'' workshop at RIMS in June 2018.  We are grateful to RIMS and to our hosts Hideo Takaoka and Satoshi Masaki for the opportunity to participate in this event.

R.~K. was supported, in part, by NSF grant DMS-1600942 and M.~V. by grant DMS-1500707.  This work was also supported by the Research Institute for Mathematical Sciences, a Joint Usage/Research Center located in Kyoto University.


\end{document}